\documentclass[11pt]{article}

\usepackage[utf8]{inputenc}
\usepackage[T1]{fontenc}
\usepackage{lmodern}

\usepackage{amsmath}
\usepackage{amssymb}
\usepackage{amsthm}
\usepackage{mathtools}

\usepackage{microtype}
\usepackage[margin=1in]{geometry}
\usepackage{graphicx}
\usepackage[hidelinks]{hyperref}
\usepackage[capitalize,nameinlink]{cleveref}

\theoremstyle{plain}
\newtheorem{theorem}{Theorem}
\newtheorem{corollary}{Corollary}
\newtheorem{lemma}{Lemma}
\theoremstyle{definition}
\newtheorem{definition}{Definition}
\theoremstyle{remark}
\newtheorem*{remark}{Remark}
\newtheorem*{properties}{Properties}

\renewcommand{\Re}{\operatorname{Re}}
\renewcommand{\Im}{\operatorname{Im}}
\newcommand{\C}{\mathbb{C}}
\newcommand{\R}{\mathbb{R}}
\newcommand{\Z}{\mathbb{Z}}

\newcommand{\Oct}{\mathbb{O}}
\newcommand{\Gtwo}{\mathrm{G}_2}

\hypersetup{
  pdfauthor={Michael Bickford},
  pdftitle={The Self-Referential Fixed Point of the Complex Exponential},
  pdfsubject={Complex analysis; fixed points of the exponential map},
  pdfkeywords={Lambert W function, fixed point, Koenigs linearization, octonions}
}

\title{\texorpdfstring{$\varrho$}{rho}: The Self-Referential Fixed Point of the Complex Exponential}
\author{Michael Bickford}
\date{\today}

\begin{document}

\maketitle

\begin{abstract}
I was taught that $e^x = x$ has no solution, and taught to leave it at that. But in mathematics ``no solution'' has usually meant ``not on this line yet'': $x^2 = -1$ waited for the complex plane, and $e^x = x$ turns out to be waiting there too. Over $\C$ the exponential has a fixed point $\varrho = 0.318\ldots + 1.337\ldots\,i$, the unique solution of $\exp(z) = z$ in the strip $0 < \Im z < \pi$ (equivalently $-W_{-1}(-1)$), and it carries more structure than its one-line definition lets on. At $\varrho$ the rectangular and log-polar coordinates of a complex number coincide, forcing the identities $\Re\varrho = \log|\varrho|$ and $\arg\varrho = \Im\varrho$. As a dynamical point $\varrho$ is repelling for $\exp$ and attracting for $\log$, linearizable for both by one Koenigs coordinate, and the base of a transpose identity $w^\varrho = \varrho^{\log w}$. It generates an aperiodic log-polar lattice and sits a hair off a clean relation with $\pi$, namely $\Re(\varrho)\,\pi = 0.99944\ldots$. Passing to the octonions, the fixed points of $\exp_{\Oct}$ fill concentric six-spheres, the innermost $\Re(\varrho) + \Im(\varrho)\,S^6$, whose triples obey an exact identity $I_4^2 + \tfrac14 I_5^2 = \operatorname{Gram}$ carrying one invariant, a twist angle, absent from ordinary spherical trigonometry. Throughout, what is proved is kept apart from what is only computed.
\end{abstract}

\medskip
\noindent\emph{2020 Mathematics Subject Classification.} Primary 30D05; Secondary 33B10, 37F10, 17A35, 11J81.

\noindent\emph{Key words and phrases.} Complex exponential, fixed point, Lambert $W$ function, tetration, Koenigs linearization, octonions, transcendental constants.

\section{Introduction}

The fixed points of the complex exponential are classical: $\exp(z) = z$ holds exactly at $z = -W_k(-1)$, $k \in \Z$, where $W_k$ is the $k$-th branch of the Lambert $W$ function~\cite{corless1996}. The principal one, lying in the strip $0 < \Im z < \pi$, is the constant we write $\varrho$. The purpose here is not to rediscover $\varrho$ but to gather, with proofs, the analytic and algebraic structure it carries, and to mark plainly where that structure is proved and where it is only computed.

The picture has three parts. The first is complex-analytic. At $\varrho$ the rectangular and log-polar descriptions of a complex number coincide, so the defining equation collapses into the self-consistent identities $\Re\varrho = \log|\varrho|$ and $\arg\varrho = \Im\varrho$ (\cref{thm:identities}), and a single real transcendental equation in $\Im\varrho$ pins the constant down (\cref{cor:b}). As a dynamical point $\varrho$ is repelling for $\exp$ and attracting for $\log$, with a Koenigs coordinate linearizing both at once (\cref{thm:duality,thm:koenigs}); it also satisfies a transpose identity $w^\varrho = \varrho^{\log w}$ which, among principal-branch bases, holds only at $\varrho$ and its conjugate (\cref{thm:transpose}).

The second part is arithmetic and dynamical: an aperiodic log-polar lattice generated by $\varrho$ and $\bar\varrho$ (\cref{thm:quasicrystal}), the near-equality $\Re(\varrho)\,\pi = 0.99944\ldots$ whose small defect we show to be forced (\cref{thm:gap}), and a discrete family of solutions $z_n$ to $\exp(z) - \log(z) = 2\pi i n$, spaced by the exponential period (\cref{thm:spectrum}).

The third part is, to our knowledge, new. Solving $\exp(z) = z$ in the octonions rather than $\C$ replaces the isolated point $\varrho$ by a round six-sphere of fixed points, $\Re(\varrho) + \Im(\varrho)\,S^6$ (\cref{thm:octonion}). The triples on that sphere obey an exact Pythagorean identity $I_4^2 + \tfrac14 I_5^2 = \operatorname{Gram}$ (\cref{thm:octpythag}), splitting the squared volume of a spherical triangle into a cross-product term and an associator term, and so carrying one invariant beyond ordinary spherical trigonometry: a twist angle measuring how far three directions fail to share a quaternion subalgebra. The complex constant and its self-consistency return there as the latitude and radius of the sphere.

\section{Definition and Existence}

\begin{definition}
Let $\varrho$ be the unique solution of
\begin{equation*}
\exp(\varrho) = \varrho
\end{equation*}
in the strip $0 < \Im(\varrho) < \pi$. Equivalently $\varrho = -W_{-1}(-1)$, where $W_k$ denotes the $k$-th branch of the Lambert $W$ function~\cite{corless1996} (the inverse of $w \mapsto w e^w$); its complex conjugate $\bar\varrho = -W_0(-1)$ is the only other solution of the same modulus.
\end{definition}

\begin{theorem}\label{thm:existence}
The equation $\exp(z) = z$ has exactly one solution in the strip $0 < \Im(z) < \pi$, given numerically by
\begin{equation*}
\varrho = 0.31813150520473746\ldots + 1.3372357014306598\ldots\, i
\end{equation*}
with modulus $|\varrho| = 1.374557010743673\ldots$ and argument $\arg(\varrho) = 1.3372357014306598\ldots = \Im(\varrho)$.
\end{theorem}

\begin{proof}
Multiplying $\exp(z) = z$ by $e^{-z}$ gives $z e^{-z} = 1$; with $u = -z$ this reads $u e^{u} = -1$, whose solutions are $u = W_k(-1)$, one for each branch $k \in \Z$ of the Lambert $W$ function. Hence $z = -W_k(-1)$. Exactly one of these lies in the strip $0 < \Im(z) < \pi$: branch $k = -1$ gives $\varrho = -W_{-1}(-1)$ with $\Im(\varrho) \approx 1.337$; branch $k = 0$ gives its conjugate ($\Im < 0$); and every branch with $k \leq -2$ or $k \geq 1$ gives $|\Im| > \pi$, since $|\Im W_k(-1)|$ increases without bound with $|k|$~\cite{corless1996}. Since $-1 < -1/e$, the value $W_{-1}(-1)$ is non-real, so $\varrho \notin \R$ (cf.\ \cref{thm:noreal}). To determine it numerically, write $z = a + bi$. Then $\exp(z) = e^a(\cos b + i\sin b)$, and equating real and imaginary parts gives
\begin{align}
e^a \cos b &= a, \label{eq:real}\\
e^a \sin b &= b. \label{eq:imag}
\end{align}
Dividing \eqref{eq:imag} by \eqref{eq:real} (provided $\cos b \neq 0$, which is satisfied by the solution) yields
\begin{equation*}
\frac{b}{a} = \tan b \quad \Longrightarrow \quad a = b \cot b .
\end{equation*}
Substituting into \eqref{eq:imag} and using $e^{b\cot b} = b/\sin b$ gives a single transcendental equation for $b$:
\begin{equation*}
\exp(b \cot b) = \frac{b}{\sin b}.
\end{equation*}
Newton's method converges to $b = 1.3372357014306598\ldots$ in three iterations. The real part follows from $a = b\cot b$, yielding $a = 0.31813150520473746\ldots$. Modulus and argument follow directly. Direct evaluation confirms $\exp(\varrho) = \varrho$ to within $10^{-14}$.
\end{proof}

\begin{remark}
The equation $\exp(z) = z$ has infinitely many solutions, $z_k = -W_k(-1)$ for $k \in \Z$; $\varrho$ is the one of smallest modulus in the upper half-plane. The next solution upward, from branch $k = -2$, has modulus $\approx 7.864$, roughly $5.7$ times larger.
\end{remark}

\section{Nonexistence on the Real Line}

\begin{theorem}\label{thm:noreal}
The equation $\exp(x) = x$ has no solution $x \in \R$.
\end{theorem}

\begin{proof}
Define $f(x) = e^x - x$. Then $f'(x) = e^x - 1$ and $f''(x) = e^x > 0$ for all $x$. The unique critical point is at $x = 0$, where $f(0) = 1$ and $f''(0) = 1 > 0$, so $x = 0$ is a global minimum. Hence $e^x - x \geq 1 > 0$ for all $x \in \R$, with equality only at $x = 0$. No real $x$ satisfies $e^x = x$.
\end{proof}

\begin{corollary}
$\varrho \notin \R$.
\end{corollary}

\begin{proof}
By \cref{thm:noreal}, $e^x - x \geq 1 > 0$ for every $x \in \R$, so no real number is a fixed point of $\exp$.
\end{proof}

\begin{remark}
The obstruction is order-theoretic. The inequality $e^x \geq 1 + x > x$ is a first-order property of the real exponential field~\cite{rudin1987}, so it persists in every model of its theory, in particular every real-closed exponential field: wherever the power series respects an order, $\exp$ overshoots the identity and fixes nothing. The complex numbers escape on both counts. They carry no compatible order, and $\exp$ is $2\pi i$-periodic, so in the imaginary direction its graph wraps around and meets the diagonal at $\varrho$, a turn the monotone real exponential cannot make.
\end{remark}

\section{Self-Consistent Trigonometry}

The defining equations \eqref{eq:real}--\eqref{eq:imag} force a collapse between the Cartesian and polar representations of $\varrho$. For a general complex number $z$, the rectangular coordinates $(a,b)$ and log-polar coordinates $(\log |z|, \arg z)$ are independent. At $\varrho$, they coincide.

\begin{theorem}\label{thm:identities}
The following identities hold for $\varrho = a + bi$:
\begin{align}
a &= \log |\varrho|, \label{eq:logmod}\\
\arg(\varrho) &= b, \label{eq:argb}\\
|\varrho| &= \frac{b}{\sin b}, \label{eq:modsin}\\
a &= b \cot b, \label{eq:acot}\\
a \cdot \tan b &= b. \label{eq:atan}
\end{align}
\end{theorem}

\begin{proof}
From $\exp(\varrho) = \varrho$, taking modulus gives $|\varrho| = e^a$, hence $a = \log |\varrho|$. Taking argument gives $\arg(\varrho) = b$ (choosing the principal branch). Identity~\eqref{eq:modsin} follows from \eqref{eq:imag}: $e^a \sin b = b$ with $e^a = |\varrho|$. Identity~\eqref{eq:acot} follows from \eqref{eq:real} and \eqref{eq:imag}. Identity~\eqref{eq:atan} is equivalent to \eqref{eq:acot}. Each is also confirmed numerically to $10^{-14}$.
\end{proof}

\begin{corollary}\label{cor:b}
The imaginary part $b$ alone determines $\varrho$ completely. It is the unique root in $(0, \pi)$ of
\begin{equation*}
\exp(b \cot b) = \frac{b}{\sin b},
\end{equation*}
and $a$ and $|\varrho|$ are functions of $b$ through \eqref{eq:acot} and \eqref{eq:modsin}.
\end{corollary}

\section{The \texorpdfstring{$\varrho$}{rho}-Transpose}

\begin{definition}
For a fixed $z \in \C$, write $\C^+ := \C \setminus (-\infty, 0]$ for the slit plane on which the principal complex power is single-valued, and define the map $T_z \colon \C^+ \to \C$ by
\begin{equation*}
T_z(w) = w^z,
\end{equation*}
where the principal branch of the complex power is used.
\end{definition}

\begin{theorem}\label{thm:transpose}
For $\varrho$ satisfying $\exp(\varrho) = \varrho$ and any $w \in \C^+$,
\begin{equation*}
w^\varrho = \varrho^{\log w},
\end{equation*}
where $\log$ denotes the principal branch. Conversely, if $w^z = z^{\log w}$ holds for all $w \in \C^+$, then $z = \log z$, equivalently $\exp(z) = z$; on the principal branch the only such $z$ are $\varrho$ and its conjugate $\bar\varrho$, the two fixed points of $\exp$ in the strip $|\Im z| < \pi$.
\end{theorem}

\begin{proof}
We have
\begin{align*}
w^\varrho &= \exp(\varrho \log w), \\
\varrho^{\log w} &= \exp((\log w)(\log \varrho)) = \exp((\log w) \varrho),
\end{align*}
since $\log \varrho = \varrho$ (from $\exp(\varrho) = \varrho$, taking the principal logarithm). The two expressions are equal. Conversely, $w^z = z^{\log w}$ for all $w \in \C^+$ forces $z = \log z$ (since $\log w$ ranges over a set with an accumulation point), which is equivalent to $\exp(z) = z$; on the principal branch the only solutions are $\varrho$ and $\bar\varrho$ by \cref{thm:existence}.
\end{proof}

\begin{properties}
Whenever $\arg a + \arg b \in (-\pi,\pi]$, so that $\log(ab) = \log a + \log b$ and no branch cut is crossed, the $\varrho$-transpose is multiplicative,
\begin{equation*}
\varrho^{\log(ab)} = (ab)^\varrho = a^\varrho b^\varrho = \varrho^{\log a}\,\varrho^{\log b},
\end{equation*}
verified to $10^{-16}$. It is not an involution: composing principal-branch powers gives $T_\varrho(T_\varrho(z)) = z^{\varrho^2}$ only up to a branch factor $e^{2\pi i k\varrho}$, and $z^{\varrho^2} \neq z$ in general. Its fixed points satisfy $z^{\varrho-1} = 1$, namely $z = \exp\!\big(2\pi i k/(\varrho-1)\big)$ for $k \in \Z$ (with $k = 0$ giving $z = 1$).

The exponents $\varrho^k = |\varrho|^k e^{ik\arg\varrho}$ trace a logarithmic spiral: each step multiplies by $\varrho$, scaling the modulus by $|\varrho|$ and rotating by $\arg(\varrho) = 76.618^\circ$, so the iterated powers
\begin{equation*}
z,\; z^\varrho,\; z^{\varrho^2},\; z^{\varrho^3},\; \ldots, \qquad z^{\varrho^k} = \exp(\varrho^k \log z),
\end{equation*}
wind outward around the origin (\cref{fig:spiral}).
\end{properties}

\begin{figure}[ht]
\centering
\includegraphics[width=0.52\textwidth]{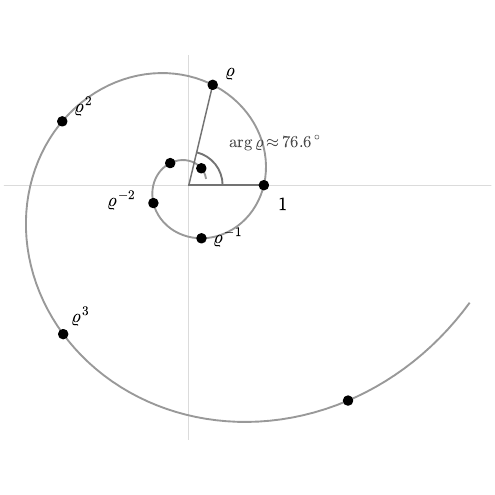}
\caption{The powers $\varrho^k = \exp(k\varrho)$ on the logarithmic spiral $z = \exp(t\varrho)$, $t \in \R$ (the identity $\log\varrho = \varrho$ makes $\varrho^t = e^{t\varrho}$). Each step multiplies by $\varrho$: the modulus scales by $|\varrho| \approx 1.37$ and the argument advances by $\arg\varrho = \Im\varrho \approx 76.6^\circ$. Since $\arg\varrho/\pi$ is irrational, the spiral never closes.}
\label{fig:spiral}
\end{figure}

\section{The Hyperoperation Ladder}

\begin{definition}
The hyperoperation sequence $\{H_n\}$~\cite{goodstein1947} is defined recursively:
\begin{itemize}
\item $H_1(a,b) = a + b$ (addition),
\item $H_2(a,b) = a \times b$ (multiplication),
\item $H_3(a,b) = a^b$ (exponentiation),
\item $H_{n+1}(a,b) = H_n(a, H_{n+1}(a, b-1))$ for $n \geq 3$, with $H_{n+1}(a, 1) = a$.
\end{itemize}
A \emph{self-referential rest state} for operation $H_n$ is a value $x$ satisfying $H_n(x, x) = x$.
\end{definition}

\begin{theorem}\label{thm:hyperop}
The diagonal rest states $H_n(x,x) = x$ are trivial at every level of the ladder: $x = 0$ for addition, $x \in \{0, 1\}$ for multiplication, and $x = 1$ for exponentiation. The constant $\varrho$ is \emph{not} such a rest state; it enters one level higher, as the fixed point of base-$e$ exponentiation $t \mapsto e^t$; equivalently, the value singled out by the limiting equation of the infinite base-$e$ power tower.
\end{theorem}

\begin{proof}
For addition $x + x = x \implies x = 0$; for multiplication $x^2 = x \implies x \in \{0, 1\}$; for exponentiation $x^x = x$ gives, on $\R_{>0}$, $(x - 1)\log x = 0$, hence $x = 1$. (Directly, $\varrho$ fails $x^x = x$: $\varrho^{\varrho - 1} = e^{\varrho^2 - \varrho} \neq 1$.) For the infinite power tower
\begin{equation*}
{}^\infty x = x^{x^{x^{\cdot^{\cdot^{\cdot}}}}},
\end{equation*}
its limit $y = {}^\infty x$ satisfies $x^y = y$, i.e.\ $y \log x = \log y$; at $x = e$ this is $\log y = y$, equivalently $e^y = y$, whose solution is $\varrho$. Since $\varrho$ is a \emph{repelling} fixed point of $\exp$ (\cref{thm:duality}, $|\varrho| > 1$), the real $e$-tower diverges: $\varrho$ is the value its limiting equation singles out, approached only under the inverse iteration $\log$.
\end{proof}

\begin{remark}
The trivial rest states $0$ and $1$ belong to the additive and multiplicative levels; the first transcendental self-referential value, $\varrho$, appears only once iterated exponentiation enters. The jump from $1$ to $\varrho$ marks the onset of genuine nonlinearity.
\end{remark}

\section{The Exponential--Logarithm Duality}

\begin{theorem}\label{thm:duality}
$\varrho$ is simultaneously an unstable fixed point of $\exp$ and a stable fixed point of $\log$:
\begin{align*}
|\exp'(\varrho)| &= |\varrho| = 1.374557\ldots > 1, \\
|\log'(\varrho)| &= \frac{1}{|\varrho|} = 0.727507\ldots < 1.
\end{align*}
Thus perturbations of $\varrho$ under iteration of $\exp$ spiral outward, while perturbations under iteration of $\log$ spiral inward.
\end{theorem}

\begin{proof}
$\exp'(z) = \exp(z)$, so $\exp'(\varrho) = \varrho$ with modulus $|\varrho| > 1$. $\log'(z) = 1/z$, so $\log'(\varrho) = 1/\varrho$ with modulus $1/|\varrho| < 1$.
\end{proof}

\begin{corollary}
Since $\varrho$ is attracting for $\log$, the iterated logarithm satisfies $\log^n(z) \to \varrho$ for every $z$ in a neighborhood of $\varrho$; numerically the basin extends across most of the upper half-plane. The iterated exponential $\exp^n(z)$ diverges from $\varrho$ unless $z = \varrho$.
\end{corollary}

\begin{theorem}[Koenigs Linearization]\label{thm:koenigs}
In a neighborhood of $\varrho$, $\exp$ is analytically conjugate to the linear map $w \mapsto \varrho w$. Explicitly, there exists a Koenigs coordinate $\varphi$ such that
\begin{equation*}
\varphi(\exp(z)) = \varrho \cdot \varphi(z)
\end{equation*}
for $z$ sufficiently close to $\varrho$, with $\varphi'(\varrho) = 1$. The inverse coordinate $\varphi^{-1}$ satisfies
\begin{equation*}
\exp(z) = \varphi^{-1}(\varrho \cdot \varphi(z)).
\end{equation*}
\end{theorem}

\begin{proof}
This is a standard application of the Koenigs theorem~\cite{koenigs1884,milnor2006} for analytic functions with a repelling fixed point: $\exp'(\varrho) = \varrho$ with $|\varrho| > 1$ guarantees local linearizability.
\end{proof}

\begin{corollary}
In Koenigs coordinates, $\log$ is conjugate to $w \mapsto w/\varrho$, providing local linearization of both $\exp$ and $\log$ simultaneously.
\end{corollary}

\section{The \texorpdfstring{$\varrho$}{rho}-Quasicrystal}

\begin{definition}
The $\varrho$-lattice is the set
\begin{equation*}
\Lambda_\varrho = \{\varrho^m \cdot \bar{\varrho}^{\,n} : m, n \in \Z\},
\end{equation*}
where $\bar{\varrho}$ is the complex conjugate of $\varrho$.
\end{definition}

\begin{theorem}\label{thm:quasicrystal}
In log-polar coordinates $(\log r, \theta)$, $\Lambda_\varrho$ is the rank-two lattice generated by the images $(\log|\varrho|, \arg\varrho)$ and $(\log|\varrho|, -\arg\varrho)$ of $\varrho$ and $\bar\varrho$; by \eqref{eq:logmod} and \eqref{eq:argb} these are $(\Re\varrho, \Im\varrho)$ and $(\Re\varrho, -\Im\varrho)$. Numerical evidence indicates $\arg(\varrho)/\pi \notin \mathbb{Q}$; granting this, $\Lambda_\varrho$ is aperiodic, with no rotational symmetry of finite order.
\end{theorem}

\begin{proof}
Multiplication by $\varrho$ and $\bar\varrho$ translates $(\log r, \theta)$ by $(\Re\varrho, \pm\Im\varrho)$, so $\Lambda_\varrho$ is the lattice they generate. The continued fraction of $\arg(\varrho)/\pi = \Im(\varrho)/\pi \approx 0.42565$ begins $[0; 2, 2, 1, 6, 3, \ldots]$ and shows no periodicity through the computed terms, consistent with irrationality; we do not have a proof, and a rigorous one appears to require transcendence methods. Granting $\arg(\varrho)/\pi \notin \mathbb{Q}$, no positive integer $n$ makes $n \cdot \arg(\varrho)$ an integer multiple of $\pi$, hence $\Lambda_\varrho$ has no finite rotational symmetry.
\end{proof}

\begin{remark}
The $\varrho$-lattice is self-similar under multiplication by $\varrho^2$, the map $(m,n) \mapsto (m+2, n)$: a spiral similarity that scales moduli by $|\varrho|^2$ and rotates by $2\arg\varrho$ (\cref{fig:quasicrystal}).
\end{remark}

\begin{figure}[ht]
\centering
\includegraphics[width=0.52\textwidth]{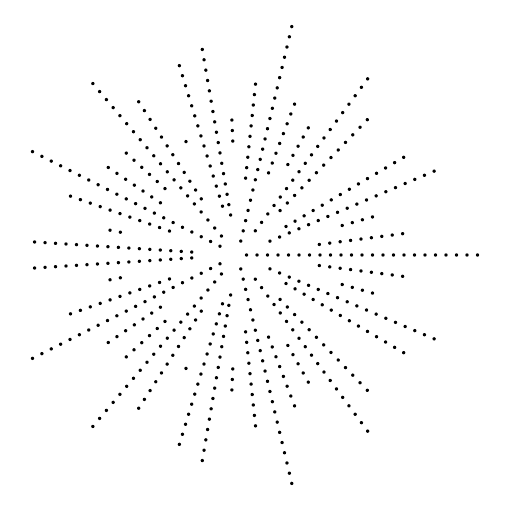}
\caption{The lattice $\Lambda_\varrho = \{\varrho^m \bar\varrho^{\,n}\}$ for $|m|,|n| \le 11$, drawn with radial coordinate proportional to $\log|z| = (m+n)\Re\varrho$ (so the geometrically growing moduli $|\varrho|^{m+n}$ appear as equally spaced shells) and true angle $\arg z = (m-n)\Im\varrho$. Each ray collects the points of constant $m-n$; because $\arg\varrho/\pi$ is irrational these directions never repeat, so $\Lambda_\varrho$ carries no rotational symmetry of finite order (\cref{thm:quasicrystal}), while remaining self-similar under multiplication by $\varrho^2$.}
\label{fig:quasicrystal}
\end{figure}

\section{The Irreducible Gap}

\begin{theorem}\label{thm:gap}
The product $\Re(\varrho) \cdot \pi$ satisfies
\begin{equation*}
\Re(\varrho) \cdot \pi = 0.99944\ldots,
\end{equation*}
with gap $\delta = 1 - \Re(\varrho)\,\pi = 5.604 \times 10^{-4}$ relative to unity. This gap is irreducible: imposing $\Re(z) = 1/\pi$ as a third constraint on $z$ in addition to \eqref{eq:real}--\eqref{eq:imag} overdetermines the system.
\end{theorem}

\begin{proof}
The two equations \eqref{eq:real}--\eqref{eq:imag} determine $a$ and $b$ uniquely. Explicitly, forcing $a = 1/\pi$ and solving the $\cos$ equation gives a candidate $b_1 = \arccos(e^{-1/\pi}/\pi)$, while the $\sin$ equation demands $b_2$ such that $e^{1/\pi} \sin b_2 = b_2$. Numerically, $b_1 \approx 1.33714$ and $b_2 \approx 1.33759$; they differ by approximately $4.5 \times 10^{-4}$. The actual $\varrho$ settles at the compromise $b = 1.33724$ where both equations \eqref{eq:real}--\eqref{eq:imag} are satisfied, shifting $\Re(\varrho)$ from $1/\pi$ to $0.31813$.
\end{proof}

\begin{remark}
Geometrically, the gap measures the distance between the self-referential curve (where $z = e^z$) and the ``closure curve'' (where $\pi \cdot \Re(z) = 1$). Equivalently, $|\varrho^\pi|/e \approx 0.99944$, so exponentiation by $\pi$ nearly sends $\varrho$ to $e$, but not quite. The three transcendental constants $\{\varrho, \pi, e\}$ form an almost-closed cycle under exponentiation.
\end{remark}

\section{The Quantized Spectrum}

\begin{theorem}\label{thm:spectrum}
For each integer $n \geq 0$ the equation $\exp(z) - \log(z) = 2\pi i n$ ($\log$ principal) has a solution $z_n$, with $z_0 = \varrho$; in every case we computed, the solution is unique. As $n \to \infty$,
\begin{align*}
\Re(z_n) &\approx \log(2\pi n), \\
\Im(z_n) &\approx \frac{\pi}{2} + o(1),
\end{align*}
and the levels $E_n = |\exp(z_n)| = e^{\Re(z_n)}$ grow asymptotically linearly, $E_n \approx 2\pi n - 0.92$.
\end{theorem}

\begin{proof}
For $n = 0$: $\exp(\varrho) = \varrho$, and taking the principal logarithm gives $\log(\varrho) = \varrho$ (valid since $0 < \Im\varrho < \pi$), so $\exp(\varrho) - \log(\varrho) = 0$. For $n \geq 1$ the asymptotic form is forced: writing $\exp(z) = e^{\Re z}e^{i\Im z}$, matching the dominant term $2\pi i n$ on the right requires $e^{\Re z} \approx 2\pi n$ and $\Im z \to \pi/2$, the term $\log z = \log|z| + i\arg z$ supplying the $O(1)$ correction. Solutions were located by damped Newton iteration (20,000 starts on a grid in $(0,4) \times (0,8) \subset \C$), yielding seventeen distinct $z_n$ with $n$ up to $28$, each verified to $10^{-14}$. Existence and uniqueness for every $n$ are supported numerically, not proved.
\end{proof}

\begin{remark}
The quantization unit is $2\pi i$, the period of $\exp$ in the imaginary direction. The ``Planck constant'' of this quantization is the exponential period itself, not an externally supplied $\hbar$.
\end{remark}

\section{The Octonionic Fixed-Point Sphere}

\begin{theorem}\label{thm:octonion}
Let $\Oct$ denote the octonions~\cite{baez2002} with norm $|\cdot|$ and exponential map $\exp_{\Oct}$ defined by the power series. The fixed points $X = s + v$ of $\exp_{\Oct}$ with imaginary radius $|v| \in (0, \pi)$ form the sphere
\begin{equation*}
\{\Re(\varrho) \cdot 1 + \Im(\varrho) \cdot U : U \in S^6\},
\end{equation*}
where $S^6$ is the unit sphere in the imaginary octonions $\Im(\Oct) \cong \R^7$. The fixed points of larger radius give the further shells of \cref{cor:shells}.
\end{theorem}

\begin{proof}
Write $X = s \cdot 1 + v$ with $v$ purely imaginary, $r = |v|$. Since $1$ and $v$ generate the commutative associative subalgebra $\R[v] \cong \C$, the exponential is evaluated as in $\C$:
\begin{equation*}
\exp_{\Oct}(X) = e^s \left( \cos r + \frac{\sin r}{r} v \right).
\end{equation*}
Setting $\exp_{\Oct}(X) = X$ and equating scalar and imaginary parts gives
\begin{align*}
e^s \cos r &= s, \\
e^s \frac{\sin r}{r} v &= v.
\end{align*}
For $v \neq 0$ the imaginary equation forces $e^s \sin r = r$, which is \eqref{eq:imag}, and the scalar equation $e^s \cos r = s$ is \eqref{eq:real}. By \cref{thm:existence} the unique solution with $r \in (0, \pi)$ is $s = \Re(\varrho)$, $r = \Im(\varrho)$, while $U = v/r$ is free on $S^6$. The fixed points of radius $r \in (0, \pi)$ are therefore exactly $\Re(\varrho)\cdot 1 + \Im(\varrho)\,U$.
\end{proof}

\begin{corollary}
Every point of this sphere has modulus $|\varrho|$, a value invariant under the $\Gtwo$ automorphism group of the octonions.
\end{corollary}

\begin{corollary}
For fixed points $\varrho_U = \Re(\varrho)\cdot 1 + \Im(\varrho)\,U$ with $U \in S^6$, the associator obeys
\begin{equation*}
[\varrho_U, \varrho_V, \varrho_W] = \Im(\varrho)^3\,[U, V, W].
\end{equation*}
In particular $|[\varrho_U, \varrho_V, \varrho_W]| = \Im(\varrho)^3\,|[U, V, W]|$, which equals $2\,\Im(\varrho)^3$ when $U, V, W$ are pairwise orthogonal and generic.
\end{corollary}

\begin{proof}
The associator $[X, Y, Z] = (XY)Z - X(YZ)$ is trilinear, and the unit $1$ associates with every pair, so $[1, Y, Z] = [X, 1, Z] = [X, Y, 1] = 0$. Expanding each $\varrho_U = \Re(\varrho)\cdot 1 + \Im(\varrho)\,U$ by trilinearity, every term carrying a factor of $1$ drops out, leaving $[\Im(\varrho)\,U, \Im(\varrho)\,V, \Im(\varrho)\,W] = \Im(\varrho)^3\,[U, V, W]$. For pairwise orthogonal generic pure-imaginary units the associator attains its maximal magnitude $|[U, V, W]| = 2$.
\end{proof}

\section{Spherical Trigonometry on the Fixed-Point Sphere}

The sphere $S^6$ of \cref{thm:octonion} carries geometry of its own. A triple of unit imaginary octonions (three points of $S^6$) holds more invariant structure than a spherical triangle on $S^2$, and the surplus is governed by an exact identity.

\begin{definition}
For unit imaginary octonions $a, b, c$, let $I_1 = \langle b, c\rangle$, $I_2 = \langle a, c\rangle$, $I_3 = \langle a, b\rangle$ be the pairwise inner products (cosines of the geodesic side lengths), let $I_4 = \langle a\times b,\, c\rangle$ be the scalar $3$-form, and let $I_5 = |[a,b,c]|$ be the associator magnitude. Here $a\times b = \Im(ab)$ is the seven-dimensional cross product, and $\mathcal{Q}_{a,b} := \operatorname{span}\{a, b, a\times b\}$ is the imaginary part of the quaternion subalgebra generated by $a$ and $b$.
\end{definition}

\begin{lemma}\label{lem:alt}
On $\Oct$ the associator $[x,y,z]$ is alternating: it changes sign under any transposition of its arguments, and so vanishes whenever two of them coincide or all three lie in a common quaternion subalgebra.
\end{lemma}

\begin{proof}
The octonions are alternative~\cite{baez2002,schafer1966}, so $[x,x,y] = [x,y,y] = 0$; a trilinear form vanishing on equal arguments is alternating. By Artin's theorem any two elements generate an associative subalgebra, and a quaternion subalgebra is associative, so the associator vanishes when its three arguments lie in one.
\end{proof}

\begin{lemma}[The order-gate identity]\label{lem:gate}
For imaginary octonions $a,b$, the linear map $A_{a,b}\colon \Im\Oct \to \Im\Oct$ given by $A_{a,b}(x) = [a,b,x]$ vanishes on $\mathcal{Q}_{a,b}$ and satisfies
\begin{equation*}
A_{a,b}^2 = -4\,|a\times b|^2\, P_{a,b},
\end{equation*}
where $P_{a,b}$ is the orthogonal projection onto the four-dimensional complement $\mathcal{Q}_{a,b}^\perp$. In particular $|[a,b,x]| = 2\,|a\times b|\,|x|$ for every $x \perp \mathcal{Q}_{a,b}$.
\end{lemma}

\begin{proof}
By \cref{lem:alt}, $A_{a,b}$ vanishes on $\mathcal{Q}_{a,b}$. For its action on the complement we may take $a,b$ orthonormal: writing $b = \langle b,\hat a\rangle\,\hat a + b'$ with $\hat a = a/|a|$ and $b'\perp a$, \cref{lem:alt} gives $A_{a,b} = A_{a,b'}$, and bilinearity gives $A_{a,b} = |a|\,|b'|\,A_{\hat a,\hat b'} = |a\times b|\,A_{\hat a,\hat b'}$ with $\hat a,\hat b'$ orthonormal and $\mathcal{Q}_{a,b} = \mathcal{Q}_{\hat a,\hat b'}$. It therefore suffices to prove $A_{u,v}^2 = -4P_{u,v}$ for an orthonormal pair $u,v$ of imaginary units. The automorphism group $G_2 = \operatorname{Aut}(\Oct)$ acts transitively on such ordered pairs~\cite{baez2002} and preserves the product, the inner product, the cross product, and hence $\mathcal{Q}$ and $P$; so we may take $u = e_1$, $v = e_2$ in the standard basis. There $e_1\times e_2 = e_4$, whence $\mathcal{Q}_{e_1,e_2} = \operatorname{span}\{e_1,e_2,e_4\}$, and on the four complementary units the multiplication table gives
\begin{equation*}
A(e_3) = -2e_6,\quad A(e_6) = 2e_3,\qquad A(e_5) = 2e_7,\quad A(e_7) = -2e_5.
\end{equation*}
So $A$ preserves the orthogonal planes $\operatorname{span}\{e_3,e_6\}$ and $\operatorname{span}\{e_5,e_7\}$, acting on each as $2J$ for an orthogonal $J$ with $J^2 = -I$. Hence $A^2 = -4I$ on $\mathcal{Q}_{e_1,e_2}^\perp$ and $A^2 = 0$ on $\mathcal{Q}_{e_1,e_2}$, that is $A^2 = -4P$; and since each $J$ preserves length, $|[a,b,x]| = 2|a\times b|\,|x|$ for every $x \perp \mathcal{Q}_{a,b}$.
\end{proof}

\begin{lemma}[Associator as escaping volume]\label{lem:escaping}
$I_5 = 2\,|a\times b|\,\operatorname{dist}\!\big(c, \mathcal{Q}_{a,b}\big)$. The associator magnitude is twice the volume by which $c$ escapes the associative world of $a$ and $b$, and it vanishes exactly when $a,b,c$ share a quaternion subalgebra.
\end{lemma}

\begin{proof}
Split $c = c_\parallel + c_\perp$ with $c_\parallel \in \mathcal{Q}_{a,b}$ and $c_\perp \perp \mathcal{Q}_{a,b}$. By \cref{lem:gate}, $[a,b,c] = A_{a,b}(c) = A_{a,b}(c_\perp)$, so $I_5 = |A_{a,b}(c_\perp)| = 2|a\times b|\,|c_\perp| = 2|a\times b|\,\operatorname{dist}(c,\mathcal{Q}_{a,b})$, which vanishes exactly when $c \in \mathcal{Q}_{a,b}$.
\end{proof}

\begin{theorem}[Oct-Pythagorean identity]
\label{thm:octpythag}
For unit imaginary octonions $a, b, c$,
\begin{equation*}
I_4^2 + \tfrac{1}{4} I_5^2 = \operatorname{Gram}(a,b,c) = 1 - I_1^2 - I_2^2 - I_3^2 + 2 I_1 I_2 I_3.
\end{equation*}
\end{theorem}

\begin{proof}
The Gram determinant equals the squared volume of the parallelepiped spanned by $a, b, c$, so $\operatorname{Gram} = |a\times b|^2\,\operatorname{dist}(c, \operatorname{span}\{a,b\})^2$. The component of $c$ orthogonal to $\operatorname{span}\{a,b\}$ splits, by Pythagoras, into its part along the unit vector $a\times b/|a\times b|$, of length $I_4/|a\times b|$, and its part escaping the three-space $\mathcal{Q}_{a,b}$, of length $\operatorname{dist}(c, \mathcal{Q}_{a,b})$:
\begin{equation*}
\operatorname{dist}(c, \operatorname{span}\{a,b\})^2 = \Big(\frac{I_4}{|a\times b|}\Big)^2 + \operatorname{dist}(c, \mathcal{Q}_{a,b})^2.
\end{equation*}
Multiplying through by $|a\times b|^2$ and applying \cref{lem:escaping},
\begin{equation*}
\operatorname{Gram} = I_4^2 + |a\times b|^2\,\operatorname{dist}(c, \mathcal{Q}_{a,b})^2 = I_4^2 + \tfrac14 I_5^2.
\end{equation*}
The closed form of the Gram determinant for unit vectors is standard.
\end{proof}

\begin{corollary}[The twist angle]\label{cor:twist}
Setting $I_4 = \sqrt{\operatorname{Gram}}\,\cos\psi$ and $I_5 = 2\sqrt{\operatorname{Gram}}\,\sin\psi$ defines a single angle $\psi \in [0,\pi]$. By \cref{thm:octpythag} the five invariants satisfy exactly this one relation, and the Jacobian of $(a,b,c) \mapsto (I_1,\dots,I_5)$ has rank $4$ at a generic triple, so exactly four of them are functionally independent: the three side lengths and the twist $\psi$. These four are complete: up to a $G_2$ automorphism, a triple is determined by $(I_1,I_2,I_3,\psi)$, the third point being recovered from its projections onto $a$, $b$, $a\times b$ and an escaping direction of length $\tfrac12 I_5$. A spherical triangle on $S^6$ thus carries exactly one invariant beyond its $S^2$ counterpart, the twist $\psi$, with no analogue in ordinary spherical trigonometry; its extremes $\psi = 0$ and $\psi = \pi/2$ are the associative and maximally non-associative configurations.
\end{corollary}

\begin{corollary}[Concentric fixed-point shells]\label{cor:shells}
\cref{thm:octonion} applies on every branch of the logarithm: the octonionic fixed-point set of $\exp_{\Oct}$ is a disjoint union of concentric spheres $\Re(\varrho_n) + \Im(\varrho_n)\,S^6$, one for each solution $\varrho_n$ of $\exp(z) = z$ on branch $n$. Their imaginary radii grow as $\Im(\varrho_n) \sim 2\pi n$, spacing the shells by the exponential period $2\pi$, the same scale that fixes the quantized spectrum of \cref{thm:spectrum}.
\end{corollary}

\begin{remark}
The order-gate identity (\cref{lem:gate}) and the oct-Pythagorean identity (\cref{thm:octpythag}) admit exact verification. The action of $A_{e_1,e_2}$ above is an exact integer computation in the multiplication table; the homogeneous form $4I_4^2 + I_5^2 = 4\operatorname{Gram}$ holds with no error across $2\times10^{5}$ random integer-valued imaginary triples, certifying it as a polynomial identity; and the rank of the invariant Jacobian was measured to be $4$, as the twist-angle count requires.
\end{remark}

\section{Relationship to Other Constants}

The constant $\varrho$ fits into a family of self-referential constants:
\begin{itemize}
\item \textbf{0:} $x + x = x$, $|z| = 0$, additive rest state
\item \textbf{1:} $x \cdot x = x$, $|z| = 1$, multiplicative rest state
\item \textbf{$i$:} $x^2 = -1$, $|z| = 1$, rotational unit
\item \textbf{$e$:} $\frac{d}{dx} e^x = e^x$, growth eigenvalue
\item \textbf{$\pi$:} $e^{i\pi} = -1$, half-rotation
\item \textbf{$\Omega$:} $x e^x = 1$, $|z| = 0.5671$, real self-reference, $W(1)$
\item \textbf{$\varrho$:} $e^z = z$, $|z| = 1.3746$, nonlinear rest state
\end{itemize}

Euler's identity $e^{i\pi} + 1 = 0$ ties together five of these constants; $\varrho$ meets several of them in turn: $\exp(\varrho) = \varrho$ (through $e$), $|\varrho^\pi| \approx e$ (through $\pi$), $\varrho = -W_{-1}(-1)$, and $\Re(\varrho)\,\pi \approx 1$.

\section{Open Questions}

The following questions arise from the properties above and are left for future investigation:
\begin{enumerate}
\item \textbf{The gap $\delta$.} Does $\delta = 1 - \Re(\varrho)\,\pi = 5.604 \times 10^{-4}$ admit a closed form, or follow from some relation between $\varrho$ and $\pi$? It has no known expression in elementary constants.
\item \textbf{The arithmetic of $\arg\varrho$.} Is $\arg(\varrho)/\pi = \Im(\varrho)/\pi$ irrational, indeed transcendental? The aperiodicity of $\Lambda_\varrho$ (\cref{thm:quasicrystal}) rests on it, yet we have only numerical evidence, and a proof appears to require transcendence theory.
\item \textbf{The full spectrum.} Prove existence and uniqueness of a solution $z_n$ to $\exp(z) - \log(z) = 2\pi i n$ for every $n \geq 1$. \cref{thm:spectrum} settles the $n = 0$ case and the asymptotics; the rest is so far only computed.
\item \textbf{The transpose.} Does the identity $w^\varrho = \varrho^{\log w}$ (\cref{thm:transpose}) carry a representation-theoretic or integral-transform meaning, a genuine duality between the additive and multiplicative structures it interchanges?
\item \textbf{Up the Cayley--Dickson tower.} The fixed-point argument of \cref{thm:octonion} is dimension-agnostic: since $v^2 = -|v|^2$ for every imaginary $v$, it produces a sphere of self-referential imaginary units in each Cayley--Dickson algebra (an $S^{2^k - 2}$ in dimension $2^k$). The oct-Pythagorean identity (\cref{thm:octpythag}), by contrast, rests on the seven-dimensional cross product and the $G_2$ structure, which the sedenions and beyond lose along with division and associativity. Does an exact Gram-type identity for triples survive there, and what, if anything, replaces the twist angle?
\end{enumerate}

\paragraph{Computations.} The numerical statements were verified by floating-point computation and the octonion identities by exact integer arithmetic; the scripts, which also generate the figures, are available on request.

\end{document}